\documentclass[11pt,oneside]{article}
\usepackage{amssymb,amsmath,amsthm}
\usepackage{url}
\usepackage[english]{babel}
\usepackage{pstricks,pst-node}
\renewcommand{\le}{\leqslant}
\renewcommand{\ge}{\geqslant}

\renewcommand{\mod}{\mathrm{mod\;\, }}
\newcommand{\RR}{\mathbb{R}}
\newcommand{\ZZ}{\mathbb{Z}}
\newcommand{\QQ}{\mathbb{Q}}
\newcommand{\CC}{\mathbb{C}}

\newcommand{\NN}{\mathbb{N}}
\newcommand{\R}{\mathbb{R}}
\newcommand{\Z}{\mathbb{Z}}
\newcommand{\Q}{\mathbb{Q}}

\newcommand{\N}{\mathbb{N}}
\newcommand{\DD}{\mathbf{D}}

\newcommand{\RRR}{\mathcal{R}}

\newtheorem{lemma}{Lemma}
\newtheorem{theorem}{Theorem}
\newtheorem{proposition}{Proposition}
\newtheorem{definition}{Definition}
\newtheorem{problem}{Problem}
\newtheorem{remark}{Remark}

\newtheorem{theorembz}{Theorem BZ2}

\newtheorem{theorembzz}{Theorem BZ1}

\newtheorem{corollary}{Corollary}

\newcommand{\hidden}[1]{}

\begin{document}


\title{On generalized Thue-Morse functions and their values.}

\author{
 Dzmitry Badziahin\footnote{Research partially supported by EPSRC  Grant
 EP/E061613/1}
 , Evgeny Zorin\footnote{Research partially supported by EPSRC  Grant EP/M021858/1}
}

\maketitle

\begin{abstract}
This paper naturally extends and generalizes our previous
work~\cite{badziahin_zorin_2014}, where we proved that the
Thue-Morse constant is not badly approximable. Here we consider the
Laurent series $f_d(x) = \prod_{n=0}^\infty (1 - x^{-d^n})$,
$d\in\N$, $d\geq 2$ which generalize the generating function
$f_2(x)$ of the Thue-Morse number, and study their continued
fraction expansion. In particular, we show that the convergents of
$x^{-d+1}f_d(x)$ have quite a regular structure. We address as well
the question whether the corresponding Mahler numbers
$f_d(a)\in\RR$, $a,d\in\N$, $a,d\geq 2$, are badly approximable.
\end{abstract}

\section{Introduction}

Our work on the well approximability of Thue-Morse constant~\cite{badziahin_zorin_2014} exploits the functional approximations to its generating function. 
Moreover, we show in~\cite{badziahin_zorin_2014} that the generating function of Thue-Morse
constant is rationally equivalent to Laurent series with a
relatively simple continued fraction. Naturally, it is desirable to
generalize methods from~\cite{badziahin_zorin_2014} to cover larger
classes of numbers and functions. In this paper we make a step in
this direction.

We consider the following functions
\begin{equation} \label{def_f_d}
f_d(x):= \prod_{t=0}^\infty (1-x^{-d^t}) \in \QQ((x^{-1})), \quad d\in\N, \quad d\geq 2,
\end{equation}
which generalize the
generating function, $f_2(x)$, of the Thue-Morse constant. We call them \emph{generalized Thue-Morse functions}.
By expanding the brackets in the infinite product~\eqref{def_f_d}, the function $f_d(x)$
defines an infinite Laurent series in $x^{-1}$ which is absolutely
convergent in the disc $|x|>1$. By substituting $x^d$ in place of
$x$ we obtain the following functional equation,
\begin{equation}\label{eq_funcf}
f_d(x^d) = \frac{xf_d(x)}{x-1}.
\end{equation}
Like in the classical case of real numbers, one can apply the
continued fraction algorithm to Laurent series from $\QQ((x^{-1}))$,
we discuss this in more details in Section~\ref{sec_cf}. In
particular, we can construct the continued fraction for $f_d(x)$. Its
properties were investigated by van der Poorten and many others in a
series of
papers~\cite{poorten_mendes_allouche_1991,MFvdP1989,MFvdP1991}. They
discovered quite an irregular behaviour of the sequence of partial
quotients of $f_d(x)$, see~\cite{poorten_mendes_allouche_1991}. In this
paper we show that, on the other hand, 
the function $g_d(x)$,
\begin{equation} \label{def_g}
g_d(x):= x^{-d+1}f_d(x),
\end{equation}
which is rationally dependent with $f_d(x)$, has a pretty regular continued fraction expansion, see Theorems~\ref{th1} and~\ref{th2} below.

Note that definition~\eqref{def_g} and
functional equation~\eqref{eq_funcf} together give a Mahler type functional equation for $g_d(x)$:
\begin{equation}\label{eq3_lem3}
g_d(x^d) = \frac{g_d(x)}{x^{d^2-2d}(x-1)}.
\end{equation}

In~\cite{badziahin_zorin_2014} we established a precise recurrent
formula for the sequence of partial quotients $a_i(x)\in\Q[x]$,
$i\in\N$ of $g_2(x)$.
Here we generalize this result to get more general properties of
continued fraction of $g_d(x)$ for the other integer values $d\geq
3$. For instance, we manage to provide a nice description of
convergents to $g_d(x)$, which is done in Theorem~\ref{th1}. In
Section~\ref{sec_bad} we make this description 
completely explicit
for the values $d$ such that $f_d(x)$ is so called badly
approximable, see Theorem~\ref{th2}.
\hidden{
In particular, Theorem~\ref{th2} computes the partial
quotients of $g_d(x)$ up to multiplication by real constants. For
$d=2$ these constants were computed in~\cite{badziahin_zorin_2014}.
In Section~\ref{sec_d3} we discuss the recurrent formulae for these
constants in the case $d=3$. Finally, in
Proposition~\ref{proposition_wa} we show that in the case $d\ge 4$,
$f_d(x)$ is not badly approximable and therefore Theorem~\ref{th2}
does not apply to these values $d$.
}

Also, in this paper we investigate the question whether $f_d(a)$ is
badly approximable for given integer values of $a$ and $d$ with
$a,d\ge 2$.
Recall that a number $x\in\R$ is said to be badly approximable if
there exists a positive constant $c=c(x)>0$ such that
$$
0<\left|x-p/q\right|\ge {c}/{q^2}
$$
for all integers $p,q$ with $q\neq 0$. Equivalently, the number
$x\in\R$ is badly approximable if and only if all its partial
quotients are uniformly upper bounded.

We explain in Subsection~\ref{subsection_dgeq4} that, for a trivial
reason, $f_d(a)$ is not badly approximable for $d\geq 4$. In
Subsections~\ref{subsection_deq2} and~\ref{subsection_deq3} we
provide the results which allow to verify that for a given integer
$a\geq 2$ the numbers $f_2(a)$ and $f_3(a)$ are not badly
approximable.  In particular, these results generalize the theorem
from~\cite{badziahin_zorin_2014} concerning $f_2(a)$. They remove a
principal obstacle which did not allow to apply that theorem to the
whole set of integers $a\geq 2$, this obstacle is explained in the
discussion after Corollary~16 in~\cite{badziahin_zorin_2014}.

\section{Some definitions and preparatory results on functional continued
fractions}\label{sec_cf}

\begin{definition}
We will denote by $\|u(x)\|$ the \emph{degree} of Laurent series
$u(x)\in\QQ((x^{-1}))$, that is the biggest degree having a non-zero
coefficient in the Laurent series $u(x)$. In case if $u(x)$ is a
polynomial in $x$, this definition of degree coincides with the
classical definition of degree of a polynomial.
\end{definition}
\begin{definition}
Let $p(x)/q(x)$ be a rational function and $u(x)$ be a Laurent series. We say that an integer $c$ is \emph{the rate of approximation} of $p(x)/q(x)$ to $u(x)$ if
$$
\left\|u(x)-p(x)/q(x)\right\|=-2\|q(x)\|-c.
$$
\end{definition}
{\noindent \bf Remark} It is easy to verify with a bit of linear
algebra that for any Laurent series $u(x)$ and any $n\in\NN$ there
exist polynomials $p_n(x)$ and $q_n(x)$ such that $\deg q_n\leq n$
and the rate of approximation of  $p_n(x)/q_n(x)$ to $u(x)$ is at
least $2n-2\deg q_n+1\geq 1$.

In this paper, we will extensively use the apparatus of continued
fractions. It is well known that Laurent series admit the continued
fraction construction analogous to that for real numbers,
\begin{equation}\label{f_d_cf}
u(x) = [a_0(x); a_1(x),a_2(x),\ldots,\cdots],
\end{equation}
where $a_i(x)$ are non-zero polynomials, $i\in\N$.
\emph{The $n$-th convergent} to $u(x)$, $n\in\N$, is defined to be
the following rational function:
\begin{equation}\label{f_d_cf_finite}
p_n(x)/q_n(x) = [a_0(x); a_1(x),\ldots,a_n(x)]
\end{equation}
where the rational function $p_n(x)/q_n(x)$ is taken in its reduced
form. In some situations we will need to precise Laurent series
which we approximate, so we denote by $p_{n,u}(x)/q_{n,u}(x)$ the
$n$'th convergent to Laurent series $u(x)$. Similarly, we denote by
$a_{i,u}$ the $i$th partial quotient of Laurent series $u(x)$.

The set of convergents to Laurent series $p_n(x)/q_n(x)$, $n\in\N$,
enjoys many nice properties similar to the properties of convergents
to the real numbers. So, the rational fraction $p_n(x)/q_n(x)$
approximates $u(x)$ with the rate of approximation $\deg
a_{n+1}(x)$. Also, convergents $p_n(x)/q_n(x)$ are precisely the
rational fractions having strictly positive rate of approximation to
$u(x)$ (see~\cite[Proposition~1]{vdP}). Numerators and denominators
of consecutive convergents enjoy the following recursive relations:
\begin{equation} \label{eq_converg}
\begin{aligned}
p_{n+1}(x)&=a_{n+1}(x)p_n(x)+p_{n-1}(x),\\
q_{n+1}(x)&=a_{n+1}(x)q_n(x)+q_{n-1}(x).
\end{aligned}
\end{equation}
We refer the reader to a nice paper~\cite{vdP} by van der Poorten for a more detailed account on continued fractions of formal power series.

Note that $p_n(x)$ and $q_n(x)$ are defined up to a multiplication
by a non-zero constant. Sometimes for convenience we want to get the
convergents $\hat{p}_n(x)/\hat{q}_n(x)$ such that $\hat{q}_n(x)$ is
monic. In that case~\eqref{eq_converg} should be modified to make
sure that the resulting polynomial $\hat{q}_{n+1}(x)$ remains monic:
\begin{equation} \label{eq_converg_mon}
\begin{aligned}
\hat{p}_{n+1}(x)&=\hat{a}_{n+1}(x)\hat{p}_n(x)+\beta_{n+1}\hat{p}_{n-1}(x),\\
\hat{q}_{n+1}(x)&=\hat{a}_{n+1}(x)\hat{q}_n(x)+\beta_{n+1}\hat{q}_{n-1}(x).
\end{aligned}
\end{equation}
where we define, with $\rho_n$ denoting the leading coefficient of
$q_n(x)$,
$$
\hat{a}_{n+1}(x) = \frac{a_{n+1}(x)\rho_n}{\rho_{n+1}};\quad
\beta_{n+1} = \frac{\rho_{n-1}}{\rho_{n+1}}.
$$

Below we prove two lemmata which provide two different sources of
convergents to the function $g_d(x)$, defined by~\eqref{def_g}.
\begin{lemma}\label{lem1}
Let $h_d(x):= x^{-1}f_d(x)$. If $p(x)/q(x)$ is a convergent to
$h_d(x)$ with the rate of approximation $c$ then
$\frac{(x-1)p(x^d)}{q(x^d)}$ is a convergent to $g_d(x)$ with the
rate of approximation at least $cd-1$. Moreover, this rate of approximation is precisely $cd-1$ if and only if $(x-1)\nmid q(x)$.
\end{lemma}
\proof We simply use the following functional relation:
\begin{equation} \label{lem1_eq1_1}
h_d(x^d) =
x^{-d}f_d(x^d) = \frac{ g_d(x) }{ (x-1) }.
\end{equation}
If
\begin{equation} \label{lem1_eq1}
\left\|h_d(x) - \frac{p(x)}{q(x)}\right\| = -2\|q(x)\| -
c,
\end{equation}
then by substituting $x^d$ in place of $x$ in~\eqref{lem1_eq1} and
by using~\eqref{lem1_eq1_1} we find:
\begin{equation} \label{lem1_eq2}
\left\|h_d(x^d) - \frac{p(x^d)}{q(x^d)}\right\| =
\left\|\frac{g_d(x)}{x-1} - \frac{p(x^d)}{q(x^d)}\right\|
= -2d\|q(x)\| - dc,
\end{equation}
hence, by multiplying both sides of~\eqref{lem1_eq2} by $x-1$,
$$
\left\|g_d(x) - \frac{(x-1)p(x^d)}{q(x^d)}\right\| =
-2\|q(x^d)\| - dc+1.
$$
Note that the rate of approximation of the convergent
$\frac{(x-1)p(x^d)}{q(x^d)}$ to $g_d(x)$ exactly equals $dc-1$ as
soon as $\gcd((x-1)p(x^d), q(x^d)) = 1$. Since $p(x)$ and $q(x)$ are
coprime by the definition of a convergent, this is equivalent to
$(x-1)\nmid q(x)$.
\endproof

\begin{lemma}\label{lem2}
Let $u_d(x):= (1-x^{-1}) f_d(x)$. If $p(x)/q(x)$ is a convergent to
$u_d(x)$ with the rate of approximation $c$, then
$$
\frac{p^*(x)}{q^*(x)}:=
\frac{p(x^d)}{(1+x+x^2+\cdots+x^{d-1})q(x^d)}
$$
is a convergent to $g_d(x)$ with the rate of approximation
$d(c-1)+1$. Moreover, this rate of approximation is precisely
$d(c-1)+1$ if and only if $(x-1)\nmid p(x)$.
\end{lemma}
\proof The proof is very similar to the proof of Lemma~\ref{lem1}. We firstly
observe that
$$
u_d(x^d) = (1-x^{-d})f_d(x^d) = (1+x+\cdots+x^{d-1})g_d(x).
$$
If
$$
\left\|u_d(x) - \frac{p(x)}{q(x)}\right\| = -2\|q(x)\| - c
$$
then
$$
\left\|\frac{u_d(x^d)}{1+x+\cdots+x^{d-1}} -
\frac{p(x^d)}{(1+x+\cdots+x^{d-1})q(x^d)}\right\|
$$$$
=\left\| g_d(x) - \frac{p^*(x)}{q^*(x)}\right\| =
-2d\|q(x)\| - dc - (d-1).
$$
Finally, the equality $\|q^*(x)\| = d\|q(x)\|+ d-1$ completes the
proof of the first part of the lemma.

If $(x-1)\nmid p(x)$ then $\gcd(p(x^d), 1+x+\ldots+x^{d-1}) = 1$ and
therefore $\gcd(p^*(x), q^*(x)) = 1$. Hence the rate of
approximation of the convergent $p^*(x) / q^*(x)$ to $g_d(x)$ is
exactly $d(c-1)+1$.
\endproof

\hidden{
We can make the analogous observation as after Lemma~\ref{lem1}: if
$(x-1)\nmid p(x)$ then $\gcd(p^*(x), q^*(x)) = 1$ and then rate of
approximation of the convergent $p^*(x) / q^*(x)$ is exactly
$d(c-1)+1$.
}

In fact, two collections of convergents of $g_d(x)$ provided by
Lemmata~\ref{lem1} and~\ref{lem2} cover the set of all the
convergents of $g_d(x)$. We prove this in Theorem~\ref{th1} below.
Beforehand we need one more technical lemma.
\begin{lemma} \label{lem33}
Let functions $h_d(x)$ and $u_d(x)$ be as defined in Lemmata~\ref{lem1} and~\ref{lem2} respectively.
\begin{enumerate}
\item If $\frac{p(x)}{q(x)}$ approximates $u_d(x)$ with the rate of approximation $c$, then $\frac{p(x)}{(x-1)q(x)}$ approximates $h_d(x)$ with the rate of approximation at least $c-1$.
\item If $\frac{p(x)}{q(x)}$ approximates $h_d(x)$ with the rate of approximation $c$, then $\frac{(x-1)p(x)}{q(x)}$ approximates $u_d(x)$ with the rate of approximation at least $c-1$.
\end{enumerate}
\begin{proof}
\begin{enumerate}
\item Note that $u_d(x) = (x-1)h_d(x)$. Then we
have
\begin{equation} \label{lem3_h_d_u_d}
\begin{aligned}
\left\| h_d(x) - \frac{p(x)}{(x-1)q(x)}\right\| &=
\left\| (x-1)^{-1}\left( u_d(x) -
\frac{p(x)}{q(x)}\right)\right\| \\
&= -2 \|q(x)\| - c -1 \\
&
= -2 \|(x-1)q(x)\| - c+1,
\end{aligned}
\end{equation}
which proves the first claim.
\item Roughly speaking, we reverse the order of calculations in~\eqref{lem3_h_d_u_d}:
\begin{equation} \label{lem3_u_d_h_d}
\begin{aligned}
\left\| u_d(x) - \frac{(x-1)p(x)}{q(x)}\right\| &=
\left\| (x-1)\left( h_d(x) -
\frac{p(x)}{q(x)}\right)\right\| \\
&
=\left\| h_d(x) -
\frac{p(x)}{q(x)}\right\|+1\\
&= -2 \|q(x)\| - c + 1.
\end{aligned}
\end{equation}
This proves the second claim of the lemma, hence completes the proof.
\end{enumerate}
\end{proof}
\end{lemma}


\begin{theorem}\label{th1}
Every convergent of $g_d(x)$ is either of the form given in
Lemma~\ref{lem1} or in the form given in Lemma~\ref{lem2}. More precisely, 
let $\frac{p_{m,g_d}}{q_{m,g_d}}$, $m\in\N$, be a convergent to $g_d(x)$. Then,
\begin{enumerate}
\item If $m$ is odd, then there exists $t\in\N$ such that the $t$-th convergent $\frac{p_{t,u_d}}{q_{t,u_d}}$ to $u_d(x)$, where $u_d(x)$ is defined in Lemma~\ref{lem2}, verifies
\begin{equation} \label{th1_case_lem2}
\frac{p_{m,g_d}}{q_{m,g_d}} =
\frac{p_{t,u_d}(x^d)}{(1+x+\cdots+x^{d-1})q_{t,u_d}(x^d)}\quad\mbox{with }(x-1)\nmid
p_{t,u_d}(x)
\end{equation}
\item If $m$ is even, then there exists $t\in\N$ such that the $t$-th convergent $\frac{p_{t,h_d}}{q_{t,h_d}}$ to $h_d(x)$, where $h_d(x)$ is defined in Lemma~\ref{lem1}, verifies
\begin{equation} \label{th1_case_lem1}
\frac{p_{m,g_d}}{q_{m,g_d}} =
\frac{(x-1)p_{t,h_d}(x^d)}{q_{t,h_d}(x^d)}\quad\mbox{with }(x-1)\nmid
q_{t,h_d}(x)
\end{equation}
\end{enumerate}
\end{theorem}
\proof 
We prove by induction. One can readily verify that the first two
convergents of $g_d(x)$ are $p_{0,g_d}(x)/q_{0,g_d}(x) = 0/1$ and
$p_{1,g_d}(x)/q_{1,g_d}(x) = 1/(1+x+\cdots+x^{d-1})$. The first one
is generated by Lemma~\ref{lem1} from the convergent $0/1$ to
$h_d(x)$ and the second one is generated by Lemma~\ref{lem2} from
the convergent $1/1$ to $u_d(x)$. therefore the zeroth and the first
convergents of $g_d(x)$ satisfy~\eqref{th1_case_lem1}
and~\eqref{th1_case_lem2} respectively.

Assume that for an odd $m\in\N$ we have~\eqref{th1_case_lem2}.
Denote by $c$ the rate of approximation of $u_d(x)$ by
$p_{t,u_d}(x)/q_{t,u_d}(x)$.
Then 
by Lemma~\ref{lem2} the rate of approximation of $g_d(x)$ by
$\frac{p_{m,g_d}}{q_{m,g_d}}$  is $d(c-1)+1$. Moreover, by the
general property of continued fractions this rate of approximation
also equals to $\|a_{m+1,g_d}\|$, hence $\|a_{m+1,g_d}\|=d(c-1)+1$.
Further, \eqref{eq_converg} implies
\begin{equation} \label{th1_deg_qmp1}
\|q_{m+1, g_d}(x)\| = \|q_{m, g_d}(x)\|+\|a_{m+1, g_d}(x)\| =
\|q_{m, g_d}(x)\|+ d(c-1)+1 = d\|q_{t,u_d}(x)\| + dc.
\end{equation}
Also,~\eqref{eq_converg} together with the definition of $c$ imply
\begin{equation} \label{th1_deg_qtp1}
\|q_{t+1, u_d}(x)\| = \|q_{t, u_d}(x)\|+\|a_{t+1, u_d}(x)\| = \|q_{t, u_d}(x)\|+ c
\end{equation}
Now consider two cases, $c\geq 2$ and $c=1$.

{\bf Case 1.} $c\ge 2$. By Lemma~\ref{lem33}, $\frac{p_{t,u_d}(x)}{(x-1)q_{t,u_d}(x)}$ is a convergent to $h_d(x)$, let us say it is the $s$-th convergent to  $h_d(x)$,
\begin{equation} \label{th1_case1_eq1}
\frac{p_{s,h_d}(x)}{q_{s,h_d}(x)}=\frac{p_{t,u_d}(x)}{(x-1)q_{t,u_d}(x)}.
\end{equation}
Moreover, the numerator and denominator of
$\frac{p_{t,u_d}(x)}{(x-1)q_{t,u_d}(x)}$ are coprime, since
$(x-1)\nmid p_{t,u_d}(x)$. Therefore the rate of approximation of
$h_d(x)$ by $s$-th convergent equals $c-1$. The latter fact implies
that the next convergent $\frac{p_{s+1,h_d}(x)}{q_{s+1,h_d}(x)}$ to
$h_d(x)$ satisfies
\begin{equation} \label{th1_case1_eq1}
\|q_{s+1,h_d}(x)\| = \|q_{s,h_d}(x)\|+c-1 = \|q_{t,u_d}(x)\| + c.
\end{equation}

Note that $(x-1)\nmid q_{s+1,h_d}(x)$. Indeed,
\eqref{th1_case1_eq1} implies that $q_{s,h_d}(x)$ is divisible by $x-1$ and by a general property of continued fractions ${\rm gcd}(q_{s+1,h_d}(x),q_{s,h_d}(x))=1$.

Finally, we have by Lemma~\ref{lem1} that $\frac{(x-1)p_{s+1,h_d}(x^d)}{q_{s+1,h_d}(x^d)}$ is a convergent to
$g_d(x)$. Note that
$$
\|q_{s+1,h_d}(x^d)\| = d\|q_{t,u_d}(x)\|+dc = \|q_{m+1,g_d}(x)\|,
$$
because of~\eqref{th1_case1_eq1} and~\eqref{th1_deg_qmp1}. Therefore
$p_{m+1,g_d} / q_{m+1,g_d}$ is generated by Lemma~\ref{lem1} from
$p_{s+1,h_d}(x)/q_{s+1,h_d}(x)$ and $(x-1)\nmid q_{s+1,h_d}(x)$.

{\bf Case 2.} $c = 1$. We firstly show that there exists a
convergent $p_{s,h_d}(x)/q_{s,h_d}(x)$ to $h_d(x)$ with
\begin{equation} \label{th1_case2_eq1}
\|q_{s,h_d}(x)\|= \|q_{t,u_d}(x)\|+1.
\end{equation}
If not then consider the convergent $\tilde{p}(x) /
\tilde{q}(x)$ to $h_d(x)$ where the degree of $\tilde{q}(x)$ is the
biggest possible not exceeding $\|q_{t,u_d}(x)\|$. Define $c_1\geq 0$ by $\|\tilde{q}(x)\| =
\|q_{t,u_d}(x)\| - c_1$. 
Under assumption that there is no $s\in\N$ such that~\eqref{th1_case2_eq1} holds true, we have that the rate of approximation of $h_d(x)$ by $\tilde{p}(x)/\tilde{q}(x)$ is at least $c_1+2$. 
and it is certainly bigger than~1. Then  by Lemma~\ref{lem33},
$(x-1)\tilde{p}(x)/\tilde{q}(x)$ is a convergent to $u_d(x)$ and
$$
\left\|u_d(x)-\frac{(x-1)\tilde{p}(x)}{\tilde{q}(x)}\right\|\leq -2\|\tilde{q}(x)\| - c_1-1.
$$
The rational function $(x-1)\tilde{p}(x)/\tilde{q}(x)$ does not
coincide with $p_{t,u_d}(x)/q_{t,u_d}(x)$ because otherwise we must
have $||\tilde{q}(x)|| = ||q_{t,u_d}(x)||$ and $(x-1) \mid
p_{t,u_d}(x)$. The last condition contradicts~\eqref{th1_case_lem2}.
Therefore $\|\tilde{q}(x)\|<\|q_{t,u_d}(x)\|$ and the next
convergent of $u_d(x)$ after
$\frac{(x-1)\tilde{p}(x)}{\tilde{q}(x)}$ has the degree of the
denominator at least $\|\tilde{q}(x)\|+ c_1+1>\|q_{t,u_d}(x)\|$, so
necessarily $q_{t,u_d}$ can not be a denominator of a convergent to
$u_d$ which is absurd. The last contradiction shows that there
exists $s\in\N$ verifying~\eqref{th1_case2_eq1}.

Further, we claim that $(x-1)\nmid q_{s,h_d}(x)$. Indeed, assume
this is not the case. Then  by Lemma~\ref{lem33} we have that
$\frac{p_{s,h_d}(x)}{q_{s,h_d}(x)(x-1)^{-1}}$ is a convergent to
$u_d(x)$, moreover its rate of convergence to $u_d$ is at least 2
(because $\frac{(x-1)p_{s,h_d}(x)}{q_{s,h_d}(x)}$ has the rate of
convergence at least zero). Then use~\eqref{th1_case2_eq1} to
compare the degrees of numerators and denominators to find
$$
\frac{p_{t,u_d}(x)}{q_{t,u_d}(x)} = \frac{p_{s,h_d}(x)}{q_{s,h_d}(x)(x-1)^{-1}}.
$$
This is a contradiction, because we consider the case when the rate
of convergence of $\frac{p_{t,u_d}(x)}{q_{t,u_d}(x)}$ to $u_d(x)$ is
$c=1$.

Finally, by Lemma~\ref{lem1}, $\frac{(x-1)p_{s,h_d}(x^d)}{q_{s,h_d}(x^d)}$
is a convergent to $g_d(x)$ verifying, in view of~\eqref{th1_case2_eq1} and~\eqref{th1_deg_qmp1},
$$
\|q_{s,h_d}(x^d)\| \stackrel{\eqref{th1_case2_eq1}}= d\|q_{t,u_d}(x)\|+d
\stackrel{\eqref{th1_deg_qmp1}}= \|q_{m+1,g_d}(x)\|.
$$

Therefore in both Case~1 and Case~2 we have that $p_{m+1,g_d}(x) /
q_{m+1,g_d}(x)$ is generated by Lemma~\ref{lem1} from a convergent
$p_{s,h_d}(x)/q_{s,h_d}(x)$, for some $s\in\N$, to $h_d(x)$ such that $(x-1)\nmid
q_{s,h_d}(x)$.

Then by analogous arguments one shows that $p_{m+2,g_d}(x) / q_{m+2,g_d}(x)$
is generated by Lemma~\ref{lem2} from a convergent
$p_{t+1,u_d}(x)/q_{t+1,u_d}(x)$ of $u_d(x)$. We leave verification of the details to the interested reader as an exercise.
This completes the inductional step and hence completes the proof.
\endproof

Theorem~1 shows that all convergents of $g_d(x)$ are of a very
special form. That form allows us to compute the precise formula for
convergents of $g_d(x)$ in case $f_d(x)$ is badly approximable.

\section{Badly approximable Laurent series}\label{sec_bad}

As in the classical case of real numbers, we say that $f(x)\in
\QQ((x^{-1}))$ is badly approximable if the degree of every its
partial quotient is bounded from above by an absolute constant.
Otherwise we say that $f(x)$ is well approximable. In other terms, $f(x)$
is well approximable if its continued fraction expansion contains
partial quotients of arbitrary large degree.

We recall one standard result about well (badly) approximable series,
which counterpart in $\RR$ is classical.

\begin{proposition} \label{proposition_equivalence}
Let $f(x)\in \QQ((x^{-1}))$, $a(x), b(x)\in \QQ[x]\backslash\{0\}$. Then $f(x)$ is well (respectively badly) approximable if and only if
$g(x) := \frac{a(x)}{b(x)}f(x)$ is well (respectively badly) approximable.
\end{proposition}
\proof If $f(x)$ is well approximable then $\forall c>0$ there
exists $p(x)/q(x)$ such that
$$
\left\|f(x) - \frac{p(x)}{q(x)}\right\| < -2\|q(x)\| - c.
$$
Therefore
$$
\left\|g(x) - \frac{a(x)p(x)}{b(x)q(x)}\right\| <
-2\|b(x)q(x)\| - c+ \|a(x)\|+\|b(x)\|.
$$
Since $\|a(x)\|$ and $\|b(x)\|$ are fixed and $c$ can be made
arbitrarily large, $g(x)$ is also well approximable. The inverse
statement can be proved analogously by noting that $f(x) =
\frac{b(x)}{a(x)}g(x)$.
\endproof
The next lemma shows that the continued fraction of $g_d(x)$ verifies the following very special dichotomy: either the degrees of its partial quotients are unbounded, or, if not, all these degrees are upper bounded by $d-1$.
\begin{lemma}\label{lem3}
If $g_d(x)$ has at least one partial quotient of degree at least $d$
then $g_d(x)$ is well approximable.
\end{lemma}
\proof Assume that there exists a partial quotient of $g_d(x)$ of
degree $c\ge d$. Then there exists a convergent $p(x)/q(x)$ to
$g_d(x)$ with the rate of approximation equals $c$:
\begin{equation}\label{eq_lem3}
\left|\left| g_d(x) - \frac{p(x)}{q(x)}\right|\right| = -2||q(x)|| -
c.
\end{equation}
The idea is to find another convergent $p^+(x) / q^+(x)$ to $g_d(x)$
which has the rate of approximation $c^+ > c$. If we are able to do
this, then we apply this construction recursively to find that there
exist convergents of $g_d(x)$ with arbitrarily large rate of
approximation which in turn implies that $g_d(x)$ is well
approximable.

Substitute $x^d$ in place of $x$ to~\eqref{eq_lem3} to get
\begin{equation}\label{eq2_lem3}
\left|\left| g_d(x^d) - \frac{p(x^d)}{q(x^d)}\right|\right| =
-2||q(x^d)|| - dc.
\end{equation}
\hidden{
Then we apply the equation $g_d(x) = x^{-n+1}f_d(x)$ and the
functional equation~\eqref{eq_funcf} for $f_d(x)$ to get that
\begin{equation}\label{eq3_lem3}
g_d(x^d) = \frac{g_d(x)}{x^{d^2-2d}(x-1)}.
\end{equation}
}
Further, substitute 
the right hand side of~\eqref{eq3_lem3} in place of $g_d(x^d)$
into~\eqref{eq2_lem3} and multiply both sides by $x^{d^2-2d}(x-1)$:
$$
\left|\left| g_d(x) -
\frac{x^{d^2-2d}(x-1)p(x^d)}{q(x^d)}\right|\right| = -2||q(x^d)|| -
dc +1+d^2-2d = -2||q^+(x)|| - c^+,
$$
where $q^+(x) = q(x^d)$ and $c^+ = dc+2d-1-d^2$. One can easily
check that for $c\ge d$ we have $c^+ > c$. This finishes the proof of the
lemma.
\endproof

\begin{lemma}\label{lem4}
If $f_d(x)$ is badly approximable then all the partial quotients of
$h_d(x)$ and $u_d(x)$ are linear.
\end{lemma}
\proof Assume that $h_d(x)$ has a partial quotient of degree at least 2.
In this case there exists a convergent $p(x)/q(x)$ to $h_d(x)$ with the rate
of approximation at least 2. Then Lemma~\ref{lem1} gives
$$
\left|\left| g_d(x) - \frac{(x-1)p(x^d)}{q(x^d)}\right|\right| \le
-2||q(x^d)|| - 2d+1.
$$
Since $2d-1\ge d$, we have by Lemma~\ref{lem3} that $g_d(x)$ is well
approximable. Then Proposition~\ref{proposition_equivalence} implies
that $f_d(x)$ is well approximable as well.

Similar considerations work in the case of $u_d(x)$. If there exists a
convergent $p(x)/q(x)$ of $u_d(x)$ with the rate of approximation at
least 2 then we use Lemma~\ref{lem2} to get
$$
\left|\left| g_d(x) - \frac{p^*(x)}{q^*(x)}\right|\right| \le
-2||q^*(x)|| - d-1.
$$
Again, we have $d+1>d$ and therefore $g_d(x)$ together with $f_d(x)$
are well approximable.
\endproof

If we know that $f_d(x)$ is badly approximable then with the help of
Lemma~\ref{lem4} we can find the recurrent formula for the
convergents of $g_d(x)$. The following theorem generalizes Proposition~3.2 from~\cite{badziahin_zorin_2014}.

\begin{theorem}\label{th2}
If $f_d(x)$ is badly approximable then the monic denominators
$q_{n,g_d}(x)$ of the convergents of $g_d(x)$ satisfy the following
recurrent equations
\begin{eqnarray}
&&q_{1,g_d}(x) = x^{d-1}+\cdots + x +1;\quad q_{2,g_d}(x) = x^d+1;\nonumber\\
&&q_{2k+1,g_d}(x) = (x^{d-1}+\cdots + x+1)q_{2k,g_d}(x) +
\beta_{2k+1}q_{2k-1,g_d}(x);\quad k\in\NN\label{eq1_th2}\\
&&q_{2k+2,g_d}(x) = (x-1)q_{2k+1,g_d}(x) +
\beta_{2k+2}q_{2k,g_d}(x),\label{eq2_th2}
\end{eqnarray}
where $\beta_k$ are some rational numbers.
\end{theorem}
In other words Theorem~\ref{th2} almost completely describes the continued fraction expansion
of badly approximable functions $g_d(x)$, up to determination of rational parameters $\beta_k$.

\begin{proof}[Proof of Theorem~\ref{th2}] We assume that $f_d(x)$ is badly approximable, so by Lemma~\ref{lem4} we have that
the partial quotients $a_{n,h_d}(x)$ of $h_d(x)$ are linear which in
turn implies that the $n$th convergent $\frac{p_{n,h_d}(x)}{q_{n,h_d}(x)}$
to $h_d(x)$ has denominator of degree $n$ with the rate of
approximation 1. Hence by Lemma~\ref{lem1},
\begin{equation} \label{q_one}
\frac{(x-1)p_{n,h_d}(x^d)}{q_{n,h_d}(x^d)}
\end{equation}
is a convergent to $g_d(x)$ with the rate of convergence at least
$d-1$. The polynomial $x-1$ does not divide $q_{n,h_d}(x)$ because otherwise
$\frac{p_{n,h_d}(x^d)}{q_{n,h_d}(x^d)(x-1)^{-1}}$ is a convergent to
$g_d(x)$ with the rate of approximation at least $d$ and therefore
by Lemma~\ref{lem3}, $g_d(x)$ is well approximable, which is not
true.

Also, Lemma~\ref{lem4} implies that the $n$-th convergent
$\frac{p_{n,u_d}(x)}{q_{n,u_d}(x)}$ has denominator of degree $n$
and the rate of convergence 1. Then we infer with Lemma~\ref{lem2}
that
\begin{equation}  \label{q_two}
\frac{p^*_{n,u_d}(x)}{q^*_{n,u_d}(x)} =
\frac{p_{n,u_d}(x^d)}{(x^{d-1}+\cdots+x+1)q_{n,u_d}(x^d)}
\end{equation}
is a convergent of $g_d(x)$ with the rate of convergence 1. As
before, $x-1$ does not divide $p_{n,u_d}(x)$ because otherwise
$$
\frac{(x^{d-1}+\ldots+1)^{-1}p_{n,u_d}(x^d)}{q_{n,u_d}(x^d)}
$$
is the convergent of $g_d(x)$ with the rate of approximation at
least $n$ which is impossible. Therefore $p_{n,u_d}(x^d)$ and
$(1+x+\ldots+x^{d-1})q_{n,u_d}(x^d)$ are coprime.

So for each $k\in \ZZ\ge 0$ there exists a convergent of $g_d(x)$ of the form~\eqref{q_one},
with the denominator of degree $kd$, and another one of the form~\eqref{q_two}, with the
denominator of degree $kd+d-1$. By Theorem~\ref{th1} no other
convergents of $g_d(x)$ exist. This allows us to construct $q_{1,g_d}(x)$ and $q_{2,g_d}(x)$:
$$
q_{1,g_d}(x) = (x^{d-1}+\ldots+x+1) q_{0,u_d}(x) = x^{d-1}+\ldots+x+1;
$$$$
q_{2,g_d}(x) = g_{1,h_d}(x^d) = x^d + 1.
$$
The second line above readily follows from the fact that $\frac{1}{x+1}$ is the
first convergent of $h_d(x)$.

For the general denominators $q_{n,g_d}(x)$ we have the following
formula
$$
q_{2k+1,g_d}(x) = (x^{d-1}+\cdots+x+1)q_{k,u_d}(x^d)\quad\mbox{and}\quad
q_{2k,g_d}(x) = q_{k,h_d}(x^d).
$$
Using the formula~\eqref{eq_converg_mon} for the monic convergents
of $g_d(x)$ we have
\begin{equation} \label{eq_converg_beta}
q_{2k+1,g_d}(x) = a_{2k+1}(x)q_{2k,g_d}(x) + \beta_{2k+1}q_{2k-1,g_d}(x)
\end{equation}
where $a_{2k+1}\in\QQ[x]$ is monic and $\beta_{2k+1}\in\QQ$. By
comparing the degrees of both sides of this equation we find
$\|a_{2k+1}(x)\| = d-1$. We also have $x^{d-1}+\ldots + x + 1\mid
q_{2k+1,g_d}(x), q_{2k-1,g_d}(x)$ and
$$
\gcd(x^{d-1}+\cdots+x+1, q_{2k,g_d}(x)) \mid \gcd(q_{2k-1,g_d}(x),
q_{2k,g_d}(x)) = 1.
$$
Therefore $x^{d-1}+\cdots+x+1\mid a_{2k+1}(x)$. Since the degrees of
these two polynomials coincide and both of them are monic we
conclude $a_{2k+1}(x) = x^{d-1}+\cdots+x+1$.

Next,
\begin{equation} \label{eq_q2kp2_and_beta}
q_{2k+2,g_d}(x) = a_{2k+2}(x)q_{2k+1,g_d}(x) + \beta_{2k+2}q_{2k,g_d}(x)
\end{equation}
where $a_{2k+2}(x)\in \QQ[x]$ is monic and $\beta_{2k+2}\in \QQ$.
Degree comparing gives us that $a_{2k+2}(x)$ is linear. Also we have
$$
a_{2k+2}(x)\cdot (x^{d-1}+\cdots+x+1)q_{k,u_d}(x^d) = q_{k+1,h_d}(x^d) -
q_{k,h_d}(x^d)
$$
Therefore $a_{2k+2}(x)\cdot (x^{d-1}+\cdots+x+1)$ is a polynomial of
$x^d$. This is only possible if $a_{2k+2} = x-1$.
\end{proof}

{\noindent \bf Remark.} The formulae for $q_{1,g_d}(x)$ and
$q_{2,g_d}(x)$ do not require $f_d(x)$ to be badly approximable. So
this part of Theorem~\ref{th2} is satisfied for all values~$d$.

Unfortunately, Theorem~\ref{th2} does not cover too many cases of
functions $g_d(x)$. In fact $f_d(x)$ is badly approximable only for
$d=2$ and $d=3$. For $d=2$ it is shown
in~\cite{badziahin_zorin_2014} and for $d=3$ it is shown
in~\cite{poorten_mendes_allouche_1991}. It is not too difficult to
show that $f_d(x)$ is well approximable for $d\ge 4$, however for
the sake of completeness we provide the proof here.

\begin{proposition} \label{proposition_wa}
Let $d\in\N$, $d\geq 4$. The function $f_d(x)$ is well approximable.
\end{proposition}
\proof The finite products $r_k(x) = \prod_{t=0}^k(1-x^{-d^t})$
provide approximations to $f_d(x)$ good enough to conclude that it
is well approximable. Indeed,
$$
||f_d(x) - r_k(x)|| = \left|\left|\prod_{t=0}^{k} (1-x^{-d^t})\cdot
\left(\prod_{t=1}^{\infty}(1-x^{-d^{t+k}}) - 1\right)\right|\right|
= -d^{k+1}.
$$
$r_k(x)$ is a rational function with denominator
$$
x^{\sum_{t=0}^k d^t} = x^{\frac{d^{k+1}-1}{d-1}}.
$$
Finally we have that for $d\geq 4$,  $d^{k+1} -
2\frac{d^{k+1}-1}{d-1} \to \infty$ as $k$ tends to infinity.
Therefore the rational functions $r_k(x)$ provide approximations to
$f_d(x)$ with an arbitrarily large rate. This completes the proof of
the proposition.
\endproof

\section{Computing the values of $\beta_k$}\label{sec_d3}

To find the precise formula for the continued fraction of $g_d(x)$
we still need to compute the values of the parameters $\beta_k$
in~\eqref{eq1_th2} and~\eqref{eq2_th2}. For $d=2$ this has already
been done in~\cite{badziahin_zorin_2014}.

\begin{theorembzz}
In the case $d=2$ the values $\beta_k$ in~\eqref{eq1_th2}
and~\eqref{eq2_th2} can be computed by the following recurrent
formulae
\begin{eqnarray*}
&&\beta_3 = -1,\quad \beta_4=1,\\
&&\beta_{2k+1} = -\frac{\beta_{k+1}}{\beta_{2k}},\\
&&\beta_{2k+2} = 1+(-1)^k - \beta_{2k+1}\quad\mbox{for }k\ge 2.
\end{eqnarray*}
\end{theorembzz}

In the case $d=3$ the formulae for $\beta_k$ are more complicated.
In this section we will get several equations between different
values of the sequence $\beta_k$ and will explain how to get other
equations which will finally enable us to provide the complete
recurrent formula for $\beta_k$.

From now on we will most often speak about the convergents of
$g_d(x)$ therefore for convenience instead of $p_{n,g_d}(x)$ and
$q_{n,g_d}(x)$ we will just write $p_n(x)$ and $q_n(x)$
respectively.

\begin{lemma}\label{lem5}
For all $k\in\N$ the convergents 
to $g_3(x)$ satisfy the following formula:
$$
q_{6k}(x) = q_{2k}(x^3);\quad p_{6k}(x) = x^3(x-1)p_{2k}(x^3).
$$
\end{lemma}
\proof Note that by Theorem~\ref{th2} the degrees of $q_k(x)$
exhaust all positive integers congruent to 0 or 2 modulo 3.
Therefore $\|q_{2k}(x)\| = 3k$ and $\|q_{2k+1}(x)\| = 3k+2$, $k\in\N$.

From Theorem~\ref{th2} we also have that
\begin{equation}\label{eq1_lem5}
\left\| g_3(x) - \frac{p_{2k}(x)}{q_{2k}(x)}\right\| =
-2||q_{2k}(x)|| - ||a_{2k+1}(x)|| = -2||q_{2k}(x)|| - 2.
\end{equation}
Recall that $g_d(x)$ satisfies the functional
equation~\eqref{eq3_lem3}. In particular, $g_3(x^3) =
\frac{g_3(x)}{x^3(x-1)}$. Then the equation~\eqref{eq1_lem5} with
$x^3$ substituted in place of $x$ gives us
$$
\left\| g_3(x^3) -
\frac{p_{2k}(x^3)}{q_{2k}(x^3)}\right\| = \left\|
\frac{g_3(x)}{x^3(x-1)} -
\frac{p_{2k}(x^3)}{q_{2k}(x^3)}\right\| = -6\|q_{2k}(x)\| - 6.
$$
Multiply both sides of this equation by $x^3(x-1)$ to 
get
$$
\left\|g_3(x) -
\frac{p_{2k}(x^3)x^3(x-1)}{q_{2k}(x^3)}\right\| =
-2||q_{2k}(x^3)|| - 2.
$$
This shows that $\frac{p_{2k}(x^3)x^3(x-1)}{q_{2k}(x^3)}$ is
a convergent to $g_d(x)$. Note that the fraction $\frac{p_{2k}(x^3)x^3(x-1)}{q_{2k}(x^3)}$ is irreducible, because in the opposite case the rate of its convergence to $g_3(x)$ would be at least 3, hence Lemma~\ref{lem3} would have implied that $g_3(x)$ is not badly approximable, which is not the case~\cite{poorten_mendes_allouche_1991}.
Finally, by calculating the degree of denominator of this
convergent we conclude the proof.
\endproof

Since $q_0(x) = 1$, formulae for $q_1(x)$ and $q_2(x)$ allow us to
conclude that $\beta_2=2$.

\begin{proposition}\label{prop2}
For each integer $k\ge 0$ we have
\begin{equation}\label{eq_prop2}
\beta_{6k+6}\beta_{6k+4}\beta_{6k+2} = \beta_{2k+2}.
\end{equation}
\end{proposition}
\proof By~\eqref{eq2_th2} we have
$$
q_{2k+2}(x) = (x-1)q_{2k+1}(x)+\beta_{2k+2}q_{2k}(x).
$$
We substitute $x^3$ in place of $x$, use Lemma~\ref{lem5} and consider the resulting
equation modulo $x-1$ to get
\begin{equation}\label{eq1_prop2}
q_{6k+6}(x) \equiv \beta_{2k+2}q_{6k}(x)\pmod{x-1}.
\end{equation}

Consequent usage of formula~\eqref{eq2_th2} for $q_{6k+6}(x)$ down
to $q_{6k+2}(x)$ leads to
\begin{eqnarray*}
&& q_{6k+6}(x) = (x-1)q_{6k+5}(x) + \beta_{6k+6}q_{6k+4}(x) \equiv \beta_{6k+6}q_{6k+4}(x)\\
&& = \beta_{6k+6}\beta_{6k+4}q_{6k+2}(x) \equiv
\beta_{6k+6}\beta_{6k+4}\beta_{6k+2}q_{6k}(x)\pmod{x-1}.
\end{eqnarray*}
Hence we get $\beta_{2k+2}q_{6k}(x)\equiv
\beta_{6k+6}\beta_{6k+4}\beta_{6k+2}q_{6k}(x)\pmod{x-1}$. Finally
from Lemma~\ref{lem5}, $\gcd(x-1,q_{6k}(x))\mid
\gcd(p_{6k}(x),q_{6k}(x)) = 1$, therefore we can divide the
congruence by $q_{6k}(x)$. This finishes the proof of the
proposition.
\endproof

More relations between values of $\beta$ can be derived by
considering coefficients with the highest
degrees of $x$ in $q_k(x)$. More exactly, write the polynomials $q_k(x)$ in the
following form  (recall Theorem~\ref{th1})
\begin{eqnarray*}
&& q_{2k}(x) = q_{k,h_d}(x^3) = x^{3k} +
a_{2k}x^{3k-3}+b_{2k}x^{3k-6}+\ldots;\\
&& q_{2k+1}(x) = (x^2+x+1) q_{k,u_d}(x^3) =
(x^2+x+1)(x^{3k}+a_{2k+1}x^{3k-3}+b_{2k+1}x^{3k-6}+\ldots).
\end{eqnarray*}

\begin{proposition}
Coefficients $a_k$ and $\beta_k$, $k\in\N$, are related by the following
equations
$$
a_{6k}=0,\quad a_{2k}-a_{2k-1} = \beta_{2k}-1,\quad a_{2k+1}-a_{2k}
= \beta_{2k+1}.
$$
In particular, these equations imply
\begin{equation}\label{eq_prop3}
\sum_{i=1}^6 \beta_{6k+i} = 3.
\end{equation}
\end{proposition}
\proof Firstly, by Lemma~\ref{lem5} and Theorem~\ref{th1}, $q_{6k}(x) = q_{2k}(x^3) =
q_{k,h_d}(x^9)$, therefore the coefficient at $x^{9k-3}$ in
$q_{6k}(x)$ is zero.

Secondly, we compare the coefficients at several leading degrees of
$x$ in Equation~\eqref{eq_q2kp2_and_beta}.
$$
x^{3k}+a_{2k}x^{3k-3}+\ldots=q_{2k}(x) = (x^3-1) (x^{3k-3} +
a_{2k-1}x^{3k-6}+\ldots) + \beta_{2k} (x^{3k-3} + \ldots).
$$
Comparison of the coefficients at $x^{3k-3}$ gives us the equation
$a_{2k}-a_{2k-1} = \beta_{2k}-1$.

Thirdly, for $q_{2k+1}(x)$ we have, by using Equation~\eqref{eq_converg_beta},
$$
(x^2+x+1) (x^{3k}+a_{2k+1}x^{3k-3}+\ldots) = q_{2k+1}(x) =
(x^2+x+1)(x^{3k}+a_{2k}x^{3k-3}+\ldots)
$$$$+\beta_{2k+1}(x^2+x+1)(x^{3k-3}+\ldots)
$$
Then dividing by $x^2+x+1$ and comparing the coefficients at
$x^{3k-3}$ give us 
$a_{2k+1}-a_{2k}
= \beta_{2k+1}$.

Finally we sum up six equations of the above form to get
$$
0 = a_{6k+6} - a_{6k} = \sum_{i=1}^6(a_{6k+i}-a_{6k+i-1}) = \sum_{i=1}^6 \beta_{6k+i} - 3.
$$
\endproof

One can compare the coefficients at the preceding powers of $x$ in
the formulae~\eqref{eq1_th2} and~\eqref{eq2_th2} for $q_k(x)$ to get
the relations between $b_k,a_k$ and $\beta_k$. The result is
presented in Proposition~\ref{proposition_a_k_b_k_beta_k} below. We
leave the details of the proof to the enthusiastic reader.

\begin{proposition} \label{proposition_a_k_b_k_beta_k}
Coefficients $b_k, a_k$ and $\beta_k$ are related by the following
equations
$$
b_{6k}=0,\quad b_{2k}-b_{2k-1} = \beta_{2k}a_{2k-2}-a_{2k-1},\quad
b_{2k+1}-b_{2k} = \beta_{2k+1}a_{2k-1}.
$$
In particular, these equations imply
\begin{equation}\label{eq_prop3}
\sum_{1\le i,j \le 6\atop j-i>1} \beta_{6k+i}\beta_{6k+j} =
3+\beta_{6k}\beta_{6k+1}.
\end{equation}
\end{proposition}

By considering more coefficients we can get more equations relating
values $\beta_{6k+1}, \ldots, \beta_{6k+6}$ with the previous values
of $\beta_i$, $i\leq 6k$. However they become overwhelmingly
complicated. Perhaps one can use some tricks similar to those in
Proposition~\ref{prop2} to find simpler relations between different
values of $\beta_k$. It would be very interesting to discover such
relations.

\section{Mahler numbers}

In this section we will consider the Mahler numbers $f_d(a)$, where
$a\ge 2$ is an integer and $f_d(x)$ is the Laurent series defined
by~\eqref{def_f_d}. It appears that some of approximation
properties of these numbers can be derived from the study of the continued fraction
of the function $f_d(x)$. In this section, we investigate the
following problem:

\begin{problem}
Given $a,d\in\ZZ, a,d\ge 2$, determine whether $f_d(a)$ is badly approximable.
\end{problem}




\subsection{The case $d\geq 4$} \label{subsection_dgeq4}

Problem~A is relatively easy in the case $d\geq 4$. For this case the answer follows from a simple Proposition~\ref{proposition_wa_numbers} below. 

Recall that the exponent of irrationality of $x\in\RR$ is defined to be
the supremum of all positive real numbers $\tau$ such that the
inequality $ \left|x - \frac{p}{q}\right| < q^{-\tau} $ has
infinitely many integer solutions $p,q$ with $q\neq 0$. It is easy to verify with the definitions that the irrationality exponent of a badly approximable
number necessarily equals two.

\begin{proposition} \label{proposition_wa_numbers}
Let $d\in\N$, $d\geq 4$ and let $a\in\N$, $a\geq 2$. Then the number
$f_d(a)$ is well approximable. Moreover, the exponent of
irrationality of $f_d(a)$ is at least $d-1$.
\end{proposition}

\proof The proof is very much similar to the proof of
Proposition~\ref{proposition_wa}. With the reference to the notation
of the proof~of  Proposition~\ref{proposition_wa}, note that
the coefficient with the highest degree in the series $f_d(x) -
r_k(x)$ is 1. So substituting $a$ in place of $x$ we find
$$
\left|f_d(a) - r_k(a)\right|\leq\sum_{t=d^{k+1}}^\infty a^{-t}\leq
 \frac{2}{a^{d^{k+1}}},
$$
whilst the denominator $q_k$ of the rational fraction $r_k(a)$ is at
most $a^{\frac{d^{k+1}-1}{d-1}}$. The estimate $q_k^{-(d-1)} =
a^{-(d^{k+1}-1)}\ge 2a^{-d^{k+1}}$ proves that $f_d(a)$ has exponent
of irrationality at least $d-1\geq 3$ and so $f_d(a)$ is not badly
approximable.
\endproof

\hidden{
On the other hand, for $d=2,3$ the function $f_n(x)$ is badly
approximable, 
so necessarily $v(f_2(a)) = 2$. !!!OTHER REFERENCES!!!

}

Proposition~\ref{proposition_wa_numbers} immediately tells us that
for $d\ge 4$, $f_d(a)$ are not badly approximable for any integer
$a\geq 2$. So it remains to study Problem~A for the cases $d=2$ and
$d=3$. In these two cases exponent of irrationality of $f_d(a)$,
$a\in\N$, $a\geq 2$, $d=2,3$, is 2. For $d=2$ this is proved
in~\cite{B2011} and for $d=3$ it follows
from~\cite[Theorem~2.5]{B2015}. Therefore the solution to Problem~A
in the cases $d=2,3$ needs more subtle considerations.

\subsection{The case $d=2$} \label{subsection_deq2}

The case $d=2$ is studied in~\cite[Theorem
5.1]{badziahin_zorin_2014} where the following theorem is proved.
Recall that $a\mid\mid b$ means that $a$ divides $b$ but $a^2$ does
not.

\begin{theorembz}
Let $p_t(x)/q_t(x)$ be the convergents of the series $g_2(x)$.
Assume that there exist positive integers $n,t,p$ such that
\begin{enumerate}
\item $p$ is a prime number and $p\,||\,a^{2^n}-1$;
\item 2 is a primitive root modulo $p^2$.
\item $p\mid\mid q_t(1)$;
\item $q_t'(1)\not\equiv 0\pmod p$.
\end{enumerate}
Then $f_2(a)$ is not badly approximable.
\end{theorembz}

This theorem allows us to show that $f_2(a)$ is not badly
approximable for many integer values of $a$.
However, as explained in~\cite{badziahin_zorin_2014}, there are some
integer values $a$ that can not be covered by Theorem~BZ2. The
smallest uncovered integer is $a=15$.


Here we provide a stronger version of Theorem~BZ2, which covers the
case $a=15$ as well as many other extra values of $a$. For this
stronger statement, Theorem~\ref{th4}, we did not detect any
constraints which prevent Theorem~\ref{th4} to be applied to any
integer $a\geq 2$. So we believe that this theorem allows to prove
that $f_2(a)$ is not badly approximable for all $a\geq 2$. On the
other hand the conditions in Theorem~\ref{th4} depend on several
parameters and we do not know a general procedure which provides
these parameters for a generic $a$.

In what follows, we denote by $\Gamma(a,p^k)$, $k\in\N$, $a\in\Z$, the
multiplicative subgroup of $\ZZ/p^k\ZZ$ generated by $a$.

\begin{theorem}\label{th4}
Let $p_t(x)/q_t(x)$ be the convergents of the series $g_2(x)$.
Assume that there exist positive integers $n_0,t,p$ such that
\begin{enumerate}
\item $p$ is an odd prime number and $p\,||\,a^{2^{n_0}}-1$;
\item \label{condition_Weiferich} $|\Gamma(2,p^2)| = p|\Gamma(2,p)|$;
\item  \label{condition_2power} $q_t(a^{2^{n_0}})\equiv 0\pmod{p^2}$;
\item $q_t'(1)\not\equiv 0\pmod p$.
\end{enumerate}
Then $f_2(a)$ is not badly approximable.
\end{theorem}
{\noindent \bf Remark.} Condition~\ref{condition_Weiferich} of
Theorem~\ref{th4} is satisfied for the most of primes we know of.
More precisely, the only primes which do not satisfy this condition
are the so called Weiferich primes, i.e. the primes $p$ such that
$p^2$ divides $2^{p-1}-1$. Indeed, if $p$ is a non-Weiferich prime
then property~\ref{condition_Weiferich} of Theorem~\ref{th4} follows
from Lemma~\ref{lem_W} below. Weiferich primes were rigorously
studied. Currently only two of them are known: 1093 and 3511, and no
more Weiferich primes exist~\cite{R2004} below $3\times 10^{15}$.

%

\begin{lemma} \label{lem_W}
Let $p$ be an odd prime and assume that
\begin{equation} \label{lem_W_assumption}
2^{p-1}\not\equiv 1 \pmod{p^2}.
\end{equation}
Then $ |\Gamma(2,p^2)| = p|\Gamma(2,p)|.$
\end{lemma}
\begin{proof}
The multiplicative subgroup $H$ of $\ZZ/p^2\ZZ$ of all elements
$a\equiv 1\pmod p$ has order~$p$. Because of the small Fermat's theorem, $\Gamma(2^{p-1},p^2) < H$. Hence
Assumption~\eqref{lem_W_assumption} implies that
$\left|\Gamma(2^{p-1},p^2)\right|=p$. \hidden{ Note for every
residue $a\in\Gamma(2,p)$ we trivially have at least one residue
$\tilde{a}\in\Gamma(2,p^2)$ such that $\tilde{a}\equiv a\pmod{p}$.
Indeed, by definition, there exists a $t\in\Z_{\geq 0}$ such that
$a\equiv 2^t\pmod{p}$. Then we can simply take $\tilde{a}=2^r \mod
p^2$. Further, for every $b\in\Gamma(2^{p-1},p^2)$ we have
$\tilde{a}b\equiv a\pmod{p^2}$. As we have explained just above, in
case if $p$ is not Weiferich prime we have
$|\Gamma(2^{p-1},p^2)|=p$, so we conclude that for every
$a\in\Gamma(2,p)$ there exists precisely $p$ residues in the group
$\Gamma(2,p^2)$ congruent to $a$ modulo $p$. This readily gives
condition~\ref{condition_Weiferich} of Theorem{th4}. }

Note that, by definition, $\Gamma(2^{p-1},p^2)\subset\Gamma(2,p^2)$,
hence $p\mid\left|\Gamma(2,p^2)\right|$. At the same time, by a
reduction modulo $p$ the group $\Gamma(2,p^2)$ is mapped onto the
group $\Gamma(2,p)$, hence $\left|\Gamma(2,p)\right|$ divides
$\left|\Gamma(2,p^2)\right|$.

By the small Fermat's theorem  $\left|\Gamma(2,p)\right|\mid p-1$, so ${\rm gcd}\left(\left|\Gamma(2,p)\right|,p\right)=1$. We readily infer that $ p\left|\Gamma(2,p)\right|$ divides $\left|\Gamma(2,p^2)\right|$ and so $\left|\Gamma(2,p^2)\right|\geq p\left|\Gamma(2,p)\right|$.

On the other hand, the reduction modulo $p$ sends $\Gamma(2,p^2)$ onto $\Gamma(2,p)$ and under this map each element in $\Gamma(2,p)$ has at most $p$ preimages.  We conclude that $\left|\Gamma(2,p^2)\right|= p\left|\Gamma(2,p)\right|$ and this completes the proof of the lemma.
\end{proof}

The big part of the proof of Theorem~\ref{th4} is the same as for
Theorem~BZ2. Therefore it will be just briefly outlined here and we
refer the reader to~\cite{badziahin_zorin_2014} for the details. In
this paper we mainly focus on the part of the proof which is
specific to Theorem~\ref{th4}.


In the proof of Theorem~\ref{th4} we will need the following lemma. 
\begin{lemma} \label{lemma_Gamma}
Let $a\in\Z\setminus\{0\}$, and let $p$ be an odd prime number. If
\begin{equation}  \label{lemma_Gamma_assumption}
|\Gamma(a,p^2)| = p\,|\Gamma(a,p)|,
\end{equation}
then for each $m\in\NN$,
$$
|\Gamma(a,p^{m+1})|
= p^m|\Gamma(a,p)|.
$$
\end{lemma}
\proof
We will show that for any $m\geq 2$,
\begin{equation} \label{lemma_Gamma_intermediate_conclusion}
|\Gamma(a,p^{m+1})|
= p|\Gamma(a,p^m)|,
\end{equation}
then the lemma readily follows by induction.

Fix $m\in\NN$. To simplify the notation, we write
$$
h:=|\Gamma(a,p^{m+1})|.
$$
Then, $a^h\equiv 1 \; \mod p^{m+1}$. By reducing modulo $p^m$ we get
$a^h\equiv 1 \; \mod p^{m}$, hence
\begin{equation} \label{factoring_h}
h=s\Gamma(a,p^m)
\end{equation}
for some $s\in\NN$. At the same time, we have
$$
a^{|\Gamma(a,p^m)|}\equiv 1+t p^{m} \; \mod p^{m+1}.
$$
By raising both sides of this congruence to the power $s$ and applying~\eqref{factoring_h}, we find
\begin{equation}  \label{lemma_Gamma_congruence_one}
1\equiv (1+tp^{m})^s \; \mod p^{m+1}.
\end{equation}
By expanding brackets on the right hand side of
~\eqref{lemma_Gamma_congruence_one}, we find
$$
1\equiv 1+stp^{m} \; \mod p^{m+1},
$$
hence $p$ divides either $s$ or $t$ (or both).

If $p$ divides $s$ then~\eqref{factoring_h} implies $h\geq p
|\Gamma(a,p^m)|$. On the other hand the reduction modulo $p^m$ sends
$\Gamma(a, p^{m+1})$ onto $\Gamma(a, p^{m})$ and under this map each
element in $\Gamma(a, p^{m})$ has at most $p$ preimages. Therefore
$$
h= p |\Gamma(a,p^m)|,
$$
and this is precisely~\eqref{lemma_Gamma_intermediate_conclusion}.

Now suppose that $p$ divides $t$. In this case we have
\begin{equation}  \label{lemma_Gamma_coincide}
h=\Gamma(a,p^{m+1})=\Gamma(a,p^m).
\end{equation}
Consider the polynomial congruence
\begin{eqnarray}  \label{lemma_Gamma_lemma_Hensel}
f(x)&\equiv& 0 \; \mod p^{m+1}, \label{lemma_Gamma_lemma_Hensel}
\end{eqnarray}
where $f(x)=x^{\Gamma(a,p^m)}-1$.

Because of~\eqref{lemma_Gamma_coincide} the solutions to the congruence~\eqref{lemma_Gamma_lemma_Hensel} are precisely the elements of $\Gamma(a,p^{m+1})$ and these solutions are congruent modulo $p^{m+1}$ to
\begin{equation}  \label{lemma_Gamma_list}
1, a, \dots, a^{\Gamma(a, p^m)-1}.
\end{equation}
Note that, as the representatives of $\Gamma(a,p^m)$, the elements of the list~\eqref{lemma_Gamma_list} are pairwise distinct modulo $p^m$. 

At the same time, we easily calculate
$$
f'(x)=\Gamma(a,p^m) x^{\Gamma(a,p^m)-1}
$$
The assumption~\eqref{lemma_Gamma_assumption} readily implies that $\Gamma(a,p^m)$ is divisible by $p$ for any $m\geq 2$, so for any integer value $x$ we have
$$
f'(x)\equiv 0 \; \mod p.
$$
Then Hensel's lemma implies that for any integer $u$ that
verifies~\eqref{lemma_Gamma_lemma_Hensel} and any
$\theta=0,\dots,p-1$ the integer $u+\theta p^m$ is also a solution
to~\eqref{lemma_Gamma_lemma_Hensel}.

For any $\theta=0,\dots,p-1$ the number $a+\theta p^m$ is not
congruent modulo $p^{m+1}$ to any element of the
list~\eqref{lemma_Gamma_list}, because all representatives there are
distinct modulo $p^m$. However it contradicts the fact that all the
residues modulo $p^{m+1}$ verifying~\eqref{lemma_Gamma_lemma_Hensel}
are given in~\eqref{lemma_Gamma_list}. This contradiction proves the
lemma.
\endproof


\endproof

\hidden{

In the proof we will need the following lemma. 
\begin{lemma}
Let $a\in\Z\setminus\{0\}$, $m\in\N$ and let $p$ be a prime number. 
Then,
$$
|\Gamma(a,p^{m+1})|
= p^m|\Gamma(a,p)|.
$$
\end{lemma}
\proof
We are going to prove
\begin{equation*}
|\Gamma(a,p^{m+1})|
= p|\Gamma(a,p^m)|,
\end{equation*}
then the lemma readily follows by recurrence.

To simplify the notation, let us use the notation
$$
h=|\Gamma(a,p^{m+1})|.
$$
Then, $a^h\equiv 1 \; \mod p^{m+1}$. By reducing modulo $p^m$, we readily infer $a^h\equiv 1 \; \mod p^{m}$, hence
\begin{equation} \label{factoring_h}
h=s\Gamma(a,p^m).
\end{equation}
At the same time, we have
$$
a^{|\Gamma(a,p^m)|}\equiv 1+t p^{m} \; \mod p^{m+1},
$$
By raising both sides of this congruence to the power $s$ and applying~\eqref{factoring_h}, we find
$$
(1+tp^{m})^y
$$
\endproof
\hidden{
\proof
Let $g\in\Z$ be a primitive root modulo $p^2$. By~\cite[Lemma~4.4]{badziahin_zorin_2014} the integer $g$ is also a primitive root modulo $p^m$ for any $m\in\N$. Let $n\in\N$ be the smallest positive integer that verifies
$$
g^n\equiv a \; \mod p.
$$
Naturally, $n=(p-1)/|\Gamma(a,p)|$.
\endproof
}
}

\hidden{
In particular, this lemma implies that, if its hypothesis are satisfied and for any $s\in\NN$, the
set
$$
\{2^n \; \mod p^m\mid n\in\NN, \; 2^n\equiv 2^s\pmod p\}
$$
coincides
with the set of residues modulo $p^m$ congruent to $2^s$ modulo $p$.
}

\proof[Proof of Theorem~\ref{th4}] Firstly, since $f_2(a)$ and
$g_2(a)$ are rationally dependent, to prove the theorem it is enough to show that
$g_2(a)$ is not badly approximable.

Secondly, for each convergent $p(x)/q(x)$ of $g_2(x)$ we provide the
series of convergents $\tilde{p}_n(x)/\tilde{q}_n(x)$ of $g_2(x)$
such that
\begin{equation}\label{eq1_th4}
\tilde{p}_n(x) = \prod_{t=0}^{n-1} (x^{2^t}-1) p(x^{2^n});\quad
\tilde{q}_n(x) = q(x^{2^n}).
\end{equation}
By multiplying both $p(x)$ and $q(x)$ by some integer constant, we
can always guarantee that $p(x)$, $q(x)$ and in turn
$\tilde{p}_n(x), \tilde{q}_n(x)$ are all in $\ZZ[x]$. Moreover
(see~\cite[Lemma 4.3]{badziahin_zorin_2014}), values
$\tilde{p}_n(a)/\tilde{q}_n(a)$ provide very good (but probably not
the best) approximations to $g_2(a)$. Namely, there exists a
constant $C$ which does not depend on $n$, such that
$$
\left|g_2(a) - \frac{\tilde{p}_n(a)}{\tilde{q}_n(a)}\right| \le
\frac{C}{(\tilde{q}_n(a))^2}.
$$

Hence, to show that $g_2(a)$ is not badly approximable, it is
sufficient to find the initial convergent $p(x)/q(x)$ and $n\in\NN$
such that $\tilde{p}_n(a)$ and $\tilde{q}_n(a)$ have an arbitrarily
large common integer factor. By~\eqref{eq1_th4} and the first
condition of the theorem we already have that $p^{n-n_0}\mid
\tilde{p}_n(a)$. So we only need to show that the sequence
$\tilde{q}_n(a)$, $n\in\N$, contains elements which are divisible by
arbitrarily large powers of $p$.

For the initial convergent we choose $p_t(x)/q_t(x)$. The aim now is
to show that for each $m\in \NN$ one can find $n\in \NN$ such that
$q_t(a^{2^n})$ is divisible by $p^m$. Conditions~3 and~4 and
Hensel's lemma imply that the equation $q_t(x)=0$ has a solution
$x\in\ZZ_p$ such that
\begin{equation}\label{a_2n_zero_p2}
x\equiv a^{2^{n_0}}\pmod{p^2}.
\end{equation}
In
particular, $x\equiv 1\pmod p$. We want to show that for each
$m\in\NN$ there exists $n\in\NN$ such that
\begin{equation} \label{th4_aim}
a^{2^n}\equiv
x\pmod{p^m},
\end{equation}
which will immediately imply that $p^m\mid
\tilde{q}_n(a)$.

For every $m\in \NN$ the multiplicative group $\RRR_{p^m}^*$ of
residues modulo $p^m$ has the order $(p-1)p^{m-1}$. As the element
$a^{2^{n_0}}$ is congruent to one modulo $p$, it lies in the kernel
of the canonical projection $\RRR_{p^m}^*\to \RRR_p^*$. The
multiplicative group $\RRR^*_{p}$ of residues modulo $p$ has the
order $p-1$, so the residue $a^{2^{n_0}}$ has the order $p^{l}$ in
$\RRR^*_{p^m}$, for some $l\leq m-1$. If the value $l$ is strictly
smaller than $m-1$, then we necessarily have $a^{2^{n_0}}\equiv
1\;\mod p^2$, which contradicts the first condition of the theorem, hence the
multiplicative order of $a^{2^{n_0}}$ modulo $p^{m}$ is exactly
$p^{m-1}$ and thus the set of residues $\{a^{2^{n_0}\cdot s}\mod
p^m\!: s\in \NN, \gcd(s,p)=1\}$ coincides with the set of residues
modulo $p^m$ congruent to 1 modulo $p$ but not congruent to 1 modulo
$p^2$. So, there is an $s\in\NN$ such that
\begin{equation} \label{congr_x}
a^{2^{n_0}\cdot s}\equiv
x \;\mod p^{m}
\end{equation}
and $s\not\equiv 0\,\mod p$. Moreover, because of the congruence~\eqref{a_2n_zero_p2} we have
\begin{equation} \label{s_equiv_1_mod_p}
s\equiv 1\;\mod p
\end{equation}
The congruence~\eqref{s_equiv_1_mod_p} implies that the residue of $\theta=2^{n_0}s$ modulo $p$ lies in $\Gamma(2,p)$.

Because of Condition~\ref{condition_Weiferich} we can apply
Lemma~\ref{lemma_Gamma}. It implies that for any $m\in\N$ the group
$\Gamma(2,p^m)$ coincides with the full preimage of $\Gamma(2,p)$
under the canonical projection $\RRR_{p^m}^*\to \RRR_p^*$. In
particular,
%
%
there exists $t_m\in\N$ such that
$$
2^{t_m}\equiv 2^{n_0}s \; \mod p^{m-1}.
$$
For this $t_m$, we have
\begin{equation} \label{congr_2t}
2^{2^{t_m}}\equiv 2^{2^{n_0}s} \; \mod p^m
\end{equation}
(recall that $2^{2^{n_0}s}$ has order $p^{m-1}$ in $(\ZZ/p^m\ZZ)^*$,
because $2^{2^{n_0}}$ has order $p^{m-1}$ and $s$ is coprime to
$p$). Taking~\eqref{congr_x} and~\eqref{congr_2t} together we
conclude
$$
2^{2^{t_m}}\equiv x \; \mod p^m,
$$
which is precisely~\eqref{th4_aim}. This finishes the proof.

\hidden{
As $a^{2^{n_0}}$ has order $p^{m-1}$ modulo $p^m$ and $2^s$ may take
all possible residues modulo $p^{m-1}$ such that $2^s\equiv 1\pmod
p$, the set of residues $\{a^{2^{n_0}\cdot 2^s}\;\mod p^m\mid
s\in\NN,\, 2^s\equiv 1\pmod p\}$ coincides with the set of residues
$\{a^{2^{n_0}\cdot s}\;\mod p^m\mid s\in\NN,\, a^{2^{n_0}\cdot
s}\equiv a^{2^{n_0}}\pmod{p^2}\}$. In particular, there exists $s_1$
such that $a^{2^{n_0+s_1}}\equiv x\pmod {p^m}$.}
\endproof

Theorem~\ref{th4} provides an algorithm for showing that $f_n(a)$ is
not badly approximable for a given $a$. We firstly find $p$ such
that Conditions~1 and~2 of the theorem are satisfied. Then we try to
find the denominator of a convergent $q_t(x)$ which satisfies
Conditions~3 and~4.

Let's use Theorem~\ref{th4} for some small prime values $p$. For
$p=3$ Condition~1 is satisfied for all $a$ except $a\equiv 0\pmod 3$
and $a\equiv \pm 1\pmod 9$. Condition~2 can be easily checked. With
help of Theorem~BZ1 we find
$$
q_9(x) = (x+1)(x^8-x^6+x^2+2)
$$
which satisfies $q_9(7) \equiv 0 \pmod 9$ and $q_9'(1)\not\equiv
0\pmod 3$. It is not difficult to show that for $a\not\equiv
0,3,6,\pm1\pmod 9$ one can always find $n_0$ such that
$a^{2^{n_0}}\equiv 7\pmod 9$. Therefore Theorem~\ref{th4} states
that $f_2(a)$ is not badly approximable for all $a\not \equiv
0,3,6,\pm 1\pmod 9$.

{\bf Remark.} In~\cite{badziahin_zorin_2014} the authors were too
brave to state that $f_2(a)$ is not badly approximable for all $a$
coprime with $3$. Unfortunately they forgot about the case
$a\equiv\pm 1\pmod 9$ which violates the first condition of
Theorem~BZ2.

Using $p=5$ and $q_{11}(x)$ we can show that $f(a)$ is not badly
approximable for all $a\in\NN$ such that $a\not\equiv 0\pmod 5$ and
$a\not\equiv \pm1, \pm7\pmod{25}$.

We conducted this procedure (using a small computer program)
for some other small primes $p$. The results are presented in the
following table, where the column $x\;\mod p^2$ specifies the solution to the congruence $q_t(x)\equiv 1 \pmod{p^2}$.

\medskip \begin{tabular}{|c|c|c|p{10cm}|} \hline $p$& $q_t(x)$&
$x\;\mod p^2$& values $a$ which pass
Conditions~1 and~3\\
\hline 3&$q_9(x)$&7&$a\equiv \pm 2,\pm4\pmod{9}$\\
\hline 5&$q_{11}(x)$&11&$a\not \equiv 0\pmod{5}, a\not\equiv \pm1,\pm7\pmod{25}$\\
\hline 7&$q_{41}(x)$&15&$a\equiv \pm1\pmod{7}, a\not\equiv \pm1\pmod{49}$\\
\cline{2-3} &$q_{187}(x)$&43& \\
\hline 11&$q_{43}(x)$&34&$a\equiv \pm1\pmod{11}, a\not\equiv \pm1\pmod{11^2}$\\
\hline 13&$q_{33}(x)$&14&$a \equiv \pm1,\pm 5\pmod{13}, a\not\equiv \pm1,\pm70\pmod{13^2}$\\
\hline 17&$q_{13}(x)$&69&$a^{16} \equiv 1 \pmod{17}$, $a^{16}\not\equiv 1 \pmod{17^2}$\\
\cline{2-3} &$q_{157}(x)$&86& \\
\hline 19&$q_{19}(x)$&210&$a\equiv \pm1\pmod{19}, a\not\equiv \pm1\pmod{19^2}$\\
\hline 23&$q_{79}(x)$&277&$a \equiv \pm1\pmod{23}, a\not\equiv \pm1\pmod{23^2}$\\
\cline{2-3} &$q_{187}(x)$&254& \\
\hline 29&$q_{35}(x)$&117&$a\equiv \pm1, \pm 12\pmod{29}, a\not\equiv \pm1, \pm 41\pmod{29^2}$\\
\hline 31&$q_{29}(x)$&156&$a\!\equiv\! \pm 156, \pm 280, \pm 311, \pm 340, \pm 402 \pmod{31^2}$\\
\hline 37&$q_{21}(x)$&408&$a\equiv \pm 1, \pm 6\pmod{37}, a\not\equiv \pm1, \pm 117 \pmod{37^2}$\\
\hline
\end{tabular}

\medskip

\medskip There are only two values of $a$ below 100 which are not covered by
this table: $a = 26$ and $a=82$.

For $a=26$ we can take $p=677 = a^2+1$. Then Conditions~1 and~2 are
satisfied. Further, Conditions~3 and~4 are satisfied for
$q_{319}(x)$ which has root $x\equiv 291111\equiv
26^{2^{204}}\pmod{677^2}$ in $\QQ_{677}$ and so Theorem~\ref{th4}
implies that $f_2(26)$ is not badly approximable.


For $a=82$ we can take $p=83=a+1$. Then Conditions~1 and~2 are
satisfied. Further, Conditions~3 and~4 are satisfied for $q_{91}(x)$
which has root $x\equiv 5479 \equiv 82^{2^{56}}\pmod{83^2}$ in
$\ZZ_{83}$ and so Theorem~\ref{th4} implies that $f_2(82)$ is not
badly approximable.

We believe that for each $a$ we can carefully choose $p$ and
$q_t(x)$ such that Conditions~1 --~4 of Theorem~\ref{th4} are
satisfied.

\hidden{
\begin{remark}
In the table above, we have quite a long list of values of $a$ that pass Conditions~1 and~3 for $p=17$. There are 128 values of $a$ in this list. This long list actually gives precisely the list of $a$ such that $a^{16}$ is congruent modulo $17^2$ to one of the elements of the set $\{18, 35, 69, 137, 154, 222, 256, 273\}$.
\end{remark}
}

\subsection{The case $d=3$} \label{subsection_deq3}

In the case $d=3$ we can use methods very similar to those for the
case $d=2$. However not every convergent $p(x)/q(x)$ to $g_3(x)$
produces a nice infinite sequence of convergents to $g_3(x)$. On the
other hand some of them do, as it is shown in Lemma~\ref{lem7}
below.

\begin{lemma}\label{lem7}
Let $p_t(x)/q_t(x)$ be the sequence of the convergents of $g_3(x)$
and $d_t$ be the least common multiple of the denominators of all
rational coefficients of $p_t(x)$ and $q_t(x)$. Then for each even
$t$ the rational functions $\tilde{p}_{t,n}(x)/\tilde{q}_{t,n}(x)$
where
\begin{equation}\label{eq_lem7}
\tilde{p}_{t,n}(x):= \prod_{k=0}^{n-1} (x^{3^{k+1}}(x^{3^k}-1))
p_t(x^{3^n})\quad\mbox{and}\quad \tilde{q}_{t,n}(x):= q_t(x^{3^n}),
\end{equation}
are all convergents of $g_3(x)$. Moreover for each positive integer
$a>1$ there exists a constant $C$ independent of $n$ such that
$$
\left|g_3(a) -
\frac{d_t\tilde{p}_{t,n}(a)}{d_t\tilde{q}_{t,n}(a)}\right| \le
\frac{C}{(d_t\tilde{q}_{t,n}(a))^2}.
$$
\end{lemma}
In other words Lemma~\ref{lem7} is an analogue of Lemma~4.3
from~\cite{badziahin_zorin_2014} and it says that
$d_t\tilde{p}_{t,n}(a)/d_t\tilde{q}_{t,n}(a)$ is almost the best
rational approximation of $g_3(a)$.

\proof The first statement of the lemma follows from the successive
application of Lemma~\ref{lem5}. We proceed with the proof of the second
statement.

Denote $G(x):= g_3(x) - p_t(x) / q_t(x)$, an infinite series in
$x^{-1}$. Since $t=2t_0$ is even, Theorem~\ref{th2} implies that
$G(x)$ starts with the term $c_1x^{-6t_0-2}$ where $c_1$ is some
integer constant. Take a compact disc $\DD\subset\{x\in \CC\;:\;
|x|>1\}$ with the center at infinity inside the set of convergence
of $G(x)$ which contains the value $a$. For the sake of concretness we can take $\DD=\{x\in \CC\;:\;
|x|>\frac{1+a}{2}\}$. Then there exists a constant
$c$ such that for each $x\in \DD$, $G(x)\le cx^{-6t_0-2}$. Consider
$|G(x^{3^n})|\prod_{k=0}^{n-1}(x^{3^{k+1}}(x^{3^k}-1))$ where
$n\in\NN$. Surely $x^{3^n}$ also belongs to $\DD$ therefore, taking into account the functional relations~\eqref{eq3_lem3} for
$g_3(x)$, we find
\begin{equation}\label{eq_G}
|G(x^{3^n})|\prod_{k=0}^{n-1}(x^{3^{k+1}}(x^{3^k}-1)) = \left|g_3(x)
-\frac{p_t(x^{3^{n+1}})
\prod_{k=0}^{n-1}(x^{3^{k+1}}(x^{3^k}-1))}{q_t(x^{3^{n+1}})}\right|
\end{equation}$$
\le \frac{
c\prod_{k=0}^{n-1}(x^{3^{k+1}}(x^{3^k}-1))}{x^{3^n(6t_0+2)}}
$$
By noticing that $x^{3^k}-1\le x^{3^k}$ and comparing the powers of
$x$ at the enumerator and the denominator we get that the right hand
side of this inequality is bounded above by
$$
\frac{ c\prod_{k=0}^{n-1}(x^{3^{k+1}}(x^{3^k}-1))}{x^{3^n(6t_0+2)}}
\le \frac{c}{x^{2\cdot 3^{n+1}t_0+2}}.
$$
By substituting~\eqref{eq_lem7} into the inequality~\eqref{eq_G} we
get
$$
\left|g_3(x) - \frac{\tilde{p}_{t,n}(x)}{\tilde{q}_{t,n}(x)}\right|
\le \frac{c}{x^{2\cdot 3^{n+1}t_0+2}}.
$$

The degree of the polynomial $q_t(x)$ is $3t_0$. Thus there exists
an absolute constant $c_2$ such that for each $x\in\DD$, $|q_t(x)|
\le c_2x^{3t_0}$ which in turn implies that $|\tilde{q}_{t,n}(x)| =
|q_t(x^{3^n})|\le c_2x^{3^{n+1}t_0}$. Therefore
$$
\left|g_3(x) -
\frac{d_t\tilde{p}_{t,n}(x)}{d_t\tilde{q}_{t,n}(x)}\right|\le
\frac{c\cdot c_2^2 x^{-2}d_t^2}{(d_t\tilde{q}_{t,n}(x))^2}
$$
This implies the second statement of the lemma with $C = c\cdot
c_2^2 a^{-2}d_t^2$.
\endproof

Lemma~\ref{lem7} suggests an analogous method for checking whether
$f_3(a)$ is badly approximable as in Theorem~\ref{th4}. As soon as
we have $p\mid a^{3^{n_0}}-1$ for some prime $p$, we immediately
have from the formulae~\eqref{eq_lem7} that $p^{n-n_0} \mid
\tilde{p}_{t,n}(a)$ for all integer $n\ge n_0$ and all even $t$.
Then if we are able to show that for some fixed even $t$ the
sequence $\tilde{q}_{t,n}(a)$ contains elements which are divisible
by an arbitrarily large power of $p$ then $g_3(a)$ and in turn
$f_3(a)$ are not badly approximable. We conclude this idea in the
following theorem. Since its proof mostly repeats the steps of
Theorem~\ref{th4} we leave it for an enthusiastic reader.

\begin{theorem} \label{theo_n3}
Let, as before, $p_t(x)/q_t(x)$ be the convergents of the series
$g_3(x)$. Assume that there exit positive integers $n_0,t,p$ such
that
\begin{enumerate}
\item $p\ge 5$ is a prime number and $p\,||\,a^{3^{n_0}}-1$;
\item \label{theo_n3_condition_Weiferich} $|\Gamma(3,p^2)| = p|\Gamma(3,p)|$; 
\item $t$ is even and $q_t(a^{3^{n_0}})\equiv 0\pmod{p^2}$;
\item $q_t'(1)\not\equiv 0\pmod p$.
\end{enumerate}
Then $f_3(a)$ is not badly approximable.
\end{theorem}
{\noindent \bf Remark} Similarly to the remark to Theorem~\ref{th4},
we can note that condition~\ref{theo_n3_condition_Weiferich} of
Theorem~\ref{theo_n3} holds true for all the primes verifying
\begin{equation} \label{W_3}
3^{p-1}\not\equiv 1\pmod{p^2}.
\end{equation}
As far as the authors are aware, currently they know only two primes
failing~\eqref{W_3}, $11$ and $1006003$. It is also known that
all the other primes in the range $5\leq p <2^{32}$
verify~\eqref{W_3}, see~\cite{Montgomery1993}. So, for all primes in
the range $5\leq p < 2^{32}$ different from $11$ and $1006003$,
condition~\ref{theo_n3_condition_Weiferich} of Theorem~\ref{theo_n3}
holds true.
\begin{corollary}
The number $f_3(2)$ is not badly approximable.
Moreover, for any integer $a$ congruent modulo 49 to any number from
the set
$$
\{
2,4,8,9,11,15,16,22,23,25,29,32,36,37,39,43,44,46
\}
$$
the number $f_3(a)$ is not badly approximable.
\end{corollary}
\begin{proof}
With a bit of computational efforts we can find that Theorem~\ref{theo_n3} is
applicable with the
parameters $n_0=2$, $t=8$ and $p=7$. Indeed, in this case
$$
2^{3^2}-1\equiv 21 \; \mod 7^2,
$$
so Condition~1 of Theorem~\ref{theo_n3} is satisfied. Further,
Condition~2 of Theorem~\ref{theo_n3} is satisfied as well because of
the Remark (or alternatively it is easy to check straightforwardly
that $3$ is a primitive root modulo $7^2$).
Finally, $p_8(x)/q_8(x)$ is the convergent to $g_3(x)$ with
$$
q_8(x)=1 + x^3 + x^6 + 2x^9 + 2x^{12},
$$
hence
$$
q_8\left(2^{3^2}\right)\equiv 0 \; \mod 7^2
$$
and
$$
q_8'\left(2^{3^2}\right)\equiv 3 \; \mod 7,
$$
thus Conditions~3 and~4 of Theorem~\ref{theo_n3} are satisfied as well and we conclude that $f_3(2)$ is not badly approximable. This proves the first part of the corollary.

To prove the second part of the corollary, we also choose $t=4$, $p=7$ and choose $n_0$ according to the following table

\medskip

\begin{tabular}{|c|c|c|c|c|c|c|c|c|c|c|c|c|c|c|c|c|c|c|c|}
\hline
$a$ & 2 & 4 & 8 & 9 & 11 & 15 & 16 & 22 & 23 & 25 & 29 & 32 & 36 & 37 & 39 & 43 & 44 & 46 \\
\hline
$n_0$ & 2 & 6 & 1 & 5 & 2 & 5 & 4 & 6 & 6 & 5 & 3 & 3 & 2 & 3 & 4 & 4 & 1 & 1\\
\hline
\end{tabular}

\medskip

We leave verification of the details to the interested reader.
\end{proof}

\end{document}